\documentclass[10pt]{article}     
\usepackage{amsthm}
\usepackage{amsfonts, amssymb, amsmath, amscd, latexsym}
\usepackage{mathrsfs}
\usepackage{epsfig}
\usepackage[all]{xy}

\swapnumbers

\newtheorem{theorem}{Theorem}[section]
\newtheorem{lemma}[theorem]{Lemma}
\newtheorem{proposition}[theorem]{Proposition}
\newtheorem{remark}[theorem]{Remark}

\numberwithin{equation}{section}

\usepackage{color}

\def\N{\mathbb N}
\def\RR{\mathbb R}
\def\ds{\displaystyle}
\def\bs{\boldsymbol}
\def\R{{\boldsymbol{R}}}
\def\Q{{\boldsymbol{Q}}}
\def\O{{\boldsymbol{O}}}
\def\xl{\boldsymbol{x_l}}
\def\Eqp{\boldsymbol{S^+}}
\def\Eqm{\boldsymbol{S^-}}
\def\SEqp{\boldsymbol{D^+}}
\def\SEqm{\boldsymbol{D^-}}
\def\SEqpm{\boldsymbol{D^\pm}}

\def\stripe{\mathcal{S}}

\def\GU{$\boldsymbol{H^{u}}$\,}
\def\GS{$\boldsymbol{H^{s}}$\,}
\def\GA{$\boldsymbol{H_{0}}$\,}

\def\x{{\bf x}}
\def\eu{{\rm e}}

\DeclareMathOperator{\diver}{div}

\begin{document}
\title{On the structure of radial solutions for some quasilinear elliptic equations}
\author {Andrea Sfecci
\thanks{Dipartimento di Ingegneria Industriale e Scienze Matematiche,
Universit\`a Politecnica delle Marche, Via Brecce Bianche 12, 60131 Ancona -
Italy. Partially supported by G.N.A.M.P.A.}
}
  \maketitle
\pagestyle{myheadings}

\begin{abstract}
In this paper we study entire radial solutions for the quasilinear $p$-Laplace equation $\Delta_p u + k(x) f(u) = 0$ where $k$ is a radial positive weight and the nonlinearity behaves e.g. as $f(u)=u|u|^{q-2}-u|u|^{Q-2}$ with $q<Q$. In particular we focus our attention on solutions (positive and sign changing) which are infinitesimal at infinity, thus providing an extension of a previous result by Tang (2001).
\end{abstract}

\vspace{5mm} \textbf{Key Words: } supercritical equations, radial
solution, ground states, Fowler transformation, invariant manifold.\\
\textbf{MR Subject Classification}: 35J70, 35J10, 37J10.
   \vspace{5mm}

\section{Introduction}

In this paper we are going to discuss the structure of radial solutions of the following quasilinear elliptic equation
\begin{equation}\label{plap}
\Delta_p u(\x) + k(|\x|) f(u(\x)) = 0\,,
\end{equation}
where $\Delta_p u = \diver ( |\nabla u|^{p-2} \nabla u)$ is the so-called $p$-Laplace operator, $\x\in\RR^n$ with $n>p>1$, $k:(0,+\infty) \to \RR$ and $f:\RR\to\RR$ are $C^1$-functions.
Therefore, we will consider the following ordinary differential equation
\begin{equation}\label{plap.rad}
(r^{n-1} u'|u'|^{p-2})'  + r^{n-1} k(r) f(u) = 0\,,
\end{equation}
where, with a little abuse of notation, we have set $u(r)=u(\x)$, with $r=|\x|$, and
 $'$ denotes the derivative with respect to $r$.

We assume the following hypotheses on the function $f$:
\begin{equation}\label{howisf} \tag{\bf F}
f(u)=u|u|^{q-2} b(u)\,, \text{ with }
\begin{cases}
q>2\,, b\in C^1(\RR)\,, \\
b(u)>0 \text{ in } (-d^-,d^+)\,, & d^\pm>0\,, \\
 b(-d^-)=b(d^+)=0\,.
\end{cases}
\end{equation}
Notice that if $f(u)=u|u|^{q-2}-u|u|^{Q-2}$, with $Q>q>2$, then \eqref{howisf} is fulfilled.

The following conditions on the weight $k$ are required:
\begin{equation}\label{howisk} \tag{\bf K}
k(r)= h(r) r^\delta>0 \,, \text{ with }
\begin{cases}
h>0\,, \delta>-p\,, \\
h_0 := \lim_{r\to 0} h(r) \in (0,+\infty)\,, \\
h_\infty := \lim_{r\to \infty} h(r) \in (0,+\infty) \,, \\
\limsup_{r\to 0} h'(r)r <+\infty\,, \\
 \lim_{r\to \infty} h'(r) r^{1+\varpi} = 0\,, \quad \varpi>0\,.
\end{cases}
\end{equation}

\bigbreak

The structure of radial solutions of \eqref{plap} in the case of power-type nonlinearities $f(u)=u|u|^{q-2}$ and $k\equiv 1$ is strictly related to the following constants
$$
p_*=p \frac{n-1}{n-p}\,,\qquad  p^*=\frac{np}{n-p}\,,
$$
which are respectively known as the Serrin and the Sobolev critical exponents. Such values change when we consider a non-constant weight of the type $k(r)=r^\delta$, i.e. in the case of H\'enon equation (see e.g. \cite{BSS,BW,NiHenon}).
In this paper we will discuss the existence of solutions of \eqref{plap.rad} vanishing at infinity. In particular
we classify solutions in two classes depending on their behaviour at infinity: 
{\em fast decay solutions} which satisfy
$\lim_{r\to\infty} u'(r)r^{\frac{n-1}{p-1}}=\widetilde L\in \RR$,
 and {\em slow decay solutions} satisfying $\lim_{r\to\infty} u'(r) r^{\frac{n-1}{p-1}} =\infty$.
Notice that, setting $L=-\widetilde L \frac{n-p}{p-1}$, the formers satisfy also
$\lim_{r\to\infty} u(r) r^{\frac{n-p}{p-1}} =L\in \RR$. We will denote fast decay solutions by $v(r,L)$.
Moreover, we call {\em regular solutions} of \eqref{plap.rad}, the ones satisfying $u(0)=d\in\RR$, and we denote them by $u(r,d)$.

The problem of existence of radial solutions for equation \eqref{plap} presents a wide literature and different approaches. We address the interested reader to the following papers and the references therein. In \cite{AlaQua,DolFlo,FQT,KYY,Tang00,Tang01}, different situations with $f(u)=u|u|^{q-2}$ are considered, while the case of a {\em sign-changing} nonlinearity as $f(u)=u|u|^{Q-2}-u|u|^{q-2}$, with $Q>q$, is treated in \cite{AccPuc,CDGHM,CGHH2015,CGHY2013,SerTang}.
In \cite{DaDu,KMPT,Tang00,Tang01} nonlinearities of the type $f(u)=u|u|^{q-2}-u|u|^{Q-2}$, with $Q>q$, are considered.  
The present paper will focus on this kind of nonlinearities providing a generalizations of \cite{Tang01}. We will enter in such details below.

\medbreak

If we introduce the following change of variable, which reminds a Fowler-type transformation borrowed from \cite{BidVer} (see also \cite{BaFloPino2000POIN,Fow31,F2010JDE}), 
\begin{equation}\label{FowT}
\begin{cases}
 x_l(t)=u(r)r^{\alpha_l}\\
 y_l(t)=u'(r)|u'(r)|^{p-2}r^{\beta_l}\\
\end{cases}
\qquad
r=e^t
\end{equation}
with $l>p_*$, where
\begin{equation*}
\alpha_l=\frac{p}{l-p}>0, \qquad \beta_l = (\alpha_l+1)(p-1),
\qquad \gamma_l =\beta_l - (n-1)<0, 
\end{equation*}
then equation \eqref{plap.rad} can be written in the form of a dynamical system which is not anymore singular:
\begin{equation} \label{sist} \tag{${\rm S}_l$}
\left( \begin{array}{c} \dot{x_l} \\\dot{y_l}  \end{array}\right)
= \left( \begin{array}{cc} \alpha_l & 0 \\ 0 & \gamma_l
\end{array} \right)
\left( \begin{array}{c} x_l \\ y_l  \end{array}\right) +\left(
\begin{array}{c} y_l |y_l|^{\frac{2-p}{p-1}} \\
- g_l(x_l,t) \end{array}\right)
\end{equation}
where $g_l(x_l,t)=k(\eu^t) \, f(x_l\eu^{-\alpha_l t}) \, \eu^{\alpha_l(l-1)t}$.
In order to ensure uniqueness of the solutions, we will assume in the whole paper $p\in(1,2]$.
Such a restriction is just technical, we assume it to avoid cumbersome technicalities (cf. \cite{F2009AMPA,F2010JDE}). 
We will see how the research of regular fast decay solutions of \eqref{plap.rad} corresponds to the research of homoclinic trajectories of system \eqref{sist}, see Remark \ref{corrisp} below.

\medbreak

Here is the main result of this paper. In the statement, we present the result for regular solutions which are positive near zero (+) and the symmetric situation for solutions which are negative near zero (--).

\begin{theorem}\label{main}
Consider the differential equation \eqref{plap}.
Assume \eqref{howisk}, \eqref{howisf}. If $l:=p\frac{q+\delta}{p+\delta}>p^*$, then
\begin{description}
\item[(+)] there exists an increasing sequence $(A_k)_{k\geq0}$ of positive numbers, such that $u(r,A_k)$ is a regular fast decay solution with $k$ non degenerate zeros. In particular $u(r,A_0)$ is a regular positive fast decay solution.
Moreover, $u(r,d)$ is a regular positive slow decay solution for any $0<d<A_0$, and there is $A^*_k \in [A_{k-1},A_k)$ such that
 $u(r,d)$ is a regular slow decay solution with $k$ nondegenerate zeros whenever $A^*_k<d<A_{k}$, for any $k \ge 1$.
\item[(--)] there exists an increasing sequence $(B_k)_{k\geq0}$ of positive numbers, such that $u(r,-B_k)$ is a regular fast decay solution with $k$ non degenerate zeros. In particular $u(r,-B_0)$ is a regular negative fast decay solution.
Moreover, $u(r,-d)$ is a regular negative slow decay solution for any $0<d<B_0$, and there is $B^*_k \in [B_{k-1},B_k)$ such that
 $u(r,-d)$ is a regular slow decay solution with $k$ nondegenerate zeros whenever $B^*_k<d<B_{k}$, for any $k \ge 1$.
\end{description}
\end{theorem}

Our main theorem partially extends a result obtained by Tang in \cite[Theorem 1]{Tang01}, where the author considers equation $\Delta_p u + f(u) = 0$, i.e. \eqref{plap} with $k(r)\equiv 1$. There, existence of radial ground states is obtained
assuming $f(0)=0$, $f(u)>0$ in an interval $(0,d^+)$ and $\Phi(u)=u f'(u)/f(u)$ non-increasing in the interval of positivity of $f$. The solutions provided by Tang in \cite{Tang01} correspond to regular solutions $u(r,d)$ with $d\in(0,A_0]$ in the statement of our main theorem above. 
In this paper we extend the discussion to nodal solutions, also introducing the weight $k$.
Notice that we do not require the monotonicity assumption on the function $\Phi$, so that we can consider, e.g., the case $f(u)= u^{q_1}+u^{q_2}-u^{q_3}$ with $q_1<q_2<q_3$, but we can have $A_1^*>A_0$ in the statement of Theorem \ref{main}, which means that
 we can loose the uniqueness of the positive fast decay solution ensured by the assumption on $\Phi$, cf. \cite[Theorem 2]{Tang01}.
We underline that, in order to prove our result, the adopted techniques are completely different.

As a final remark, we recall that the case of equation \eqref{plap}, where \eqref{howisf} and\eqref{howisk} hold with $\delta=0$,
 has been investigated in \cite{DaDu} in the case $l=q<p^*$. Notice that, in Theorem \ref{main} we can consider also an {\sl apparently subcritical} situation with $q<p^*$ and $\delta<0$ such that $l=p\frac{q+\delta}{p+\delta}>p^*$.

\medbreak

The paper is organized as follows. In the next section we are going to introduce the main tools needed in order to prove our main theorem. The proof is based on the study of the invariant manifolds associated to the saddle-type equilibrium $(x,y)=(0,0)$ in \eqref{sist}.
In particular, in Section \ref{sub21} we draw the phase portrait of some systems \eqref{sist} which are autonomous, then in Section \ref{sub22}, using invariant manifold theory, we provide the needed background in the non-autonomous case. Section \ref{sec3} contains the proof of the main theorem. The proof is divided in five steps: in Section \ref{step1} we introduce a truncated problem, then in Sections \ref{step2} and \ref{step3} we prove respectively the existence of fast and slow decay solutions for this problem; in Section \ref{step4} we provide the estimates on the number of zeroes of such solutions, finally in Section \ref{step5} we prove that the solutions we have found are indeed solutions of the original problem.

\section{Introduction of invariant manifolds}

In the following subsections we will introduce the main tools we need in order to prove our main result. 

In this paper, we will denote by $\bs{x_l}(t,\tau,\Q)=(x_l(t,\tau,\Q),y_l(t,\tau,\Q))$ the trajectory of \eqref{sist} which is in $\Q$ for $t=\tau$, i.e. such that $\bs{x_l}(\tau,\tau,\Q)=\Q$.

Moreover, it is easy to verify that applying \eqref{FowT} with different values $l\neq L$ we get the following identities
\begin{eqnarray}
\label{xyll}
&&\bs{x_L}(t)= \bs{x_l}(t) \eu^{(\alpha_{L}-\alpha_l)t}\,,
\label{gll}
\\
&&g_{L}(x,t) = g_l(x \eu^{-(\alpha_{L}-\alpha_l)t},t) \eu^{[\alpha_{L}(L-1)-\alpha_l(l-1)]t}\,, 
\\
\label{Gll}
&&G_L(x,t)=G_l(x \eu^{-(\alpha_{L}-\alpha_l)t},t) \eu^{p(\alpha_L-\alpha_l)t}\,.
\end{eqnarray}
where $G_{l}(x,t) = \int_0^x g_{l}(\xi,t)\, d\xi$.

\subsection{The autonomous superlinear case}\label{sub21}

In this section we focus our attention on systems \eqref{sist} which are autonomous, i.e. such that $g_l(x,t)=g_l(x)$. 
In particular we recall some of the results contained in \cite{F2010JDE}. 
 We refer also to \cite{DalFra,FraSfe1} 
 for the picture in the classical case $p=2$.
We assume the following hypotheses on the nonlinearity $g_l$.
\begin{description}
  \item[\GA] There is $l>p_*$
   such that $g_l(x,t) \equiv g_l(x)$ is $t$-independent, satisfying
  $$  \lim_{x \to 0} \frac{g_l(x)}{x} = 0 
  \qquad \text{and}\qquad
  \lim_{|x| \to +\infty} \frac{g_l(x)}{x}=+\infty\,. $$
  Moreover $g_l(x)/|x|^{p-1}$ is an increasing function, positive for $x>0$.
\end{description}  


\begin{remark}
It can be verified, with a short computation, that if $k(r)= k_0 r^\delta$, with $k_0\in\RR$ and $f(u)=u|u|^{q-2}$ then, setting $l=p \, \frac{q+\delta}{p+\delta}$, we obtain $g_l(x,t)=x|x|^{q-2}$ which satisfies \GA. In the case of a constant function $k(r)=k_0$ we get $l=q$.
\end{remark}

Assume \GA and fix the corresponding $l>p_*$ in \eqref{sist}. The origin $\O=(0,0)$ is a saddle and admits a
 $1$-dimensional unstable manifold $M^u$ and a $1$-dimensional stable manifold $M^s$.
  Moreover, by \GA,  which implicitly gives
 $$
\lim_{x \to 0} \frac{g_l(x)}{x|x|^{p-2}} = 0 
  \qquad \text{and}\qquad
  \lim_{|x| \to +\infty} \frac{g_l(x)}{x|x|^{p-2}}=+\infty\,,
$$
we have two non-trivial critical points $\boldsymbol{P^+} =(P_x^+,P_y^+)$ with $P_x^+>0>P_y^+$
  and $\boldsymbol{P^-} =(P_x^-,P_y^-)$ with $P_x^-<0<P_y^-$.
 In particular $P_x^\pm$ solves $g_l(x)=x|x|^{p-2} \alpha_l^{p-1} |\gamma_l|$ and $P_y^\pm = \mp \left(\alpha_l |P_x^\pm|\right)^{p-1} $.

  They are stable for
 $l>p^*$,   centers for $l=p^*$ and unstable for $p_*<l<p^*$.
In particular,
\begin{equation}\label{alfagamma}
\alpha_l + \gamma_l \quad
\begin{cases}
> 0 & \text{ if } p_*<l<p^*\,,
\\
= 0 & \text{ if } l=p^*\,,
\\
< 0 & \text{ if } l>p^*\,.
\end{cases} 
\end{equation}

We define the following energy function
\begin{equation*}
H_{l}(x,y) = \frac{n-p}{p}\, xy + \frac{p-1}{p}\, |y|^{\frac{p}{p-1}} + G_{l}(t,x)\,,
\end{equation*}
which is strictly related to the Pohozaev function
\begin{equation*}
\mathcal{P}(u,u',r) = r^n \left[ \frac{n-p}{p} \, \frac{u\, u' \, |u'|^{p-2}}{r} + \frac{p-1}{p} \, |u'|^p + F(u,r)\right] 	,
\end{equation*}
where $F(u,r)= \int_0^u f(\upsilon,r) \,d\upsilon$. 
In fact we have
\begin{equation}\label{Po-rel}
\mathcal{P}(u(r),u'(r),r) = H_{p^*} ( x_{p^*}(t),y_{p^*}(t),t)= H_l(x_l(t), y_l(t), t) \eu^{-(\alpha_l+\gamma_l)t} \,.
\end{equation}
A computation gives
$$
\frac{d}{dt}H_{p^*}(x_{p^*}(t),y_{p^*}(t)) = \frac{\partial}{\partial t} G_{p^*}(t,x_{p^*}(t))\,,
$$
If \GA holds with $l=p^*$, then $G_{p^*}$ is independent of $t$, so that (${\rm S}_{p^*}$) is a Hamiltonian system and the unstable and the stable manifolds coincide: we have the existence of two homoclinc trajectories (see Figure \ref{figHham} for the phase portrait in this case). Moreover we can compute $H_{p^*}(\O)=0$ and $H_{p^*}(\bs{P^\pm})<0$.

\begin{figure}[t]
\centerline{\epsfig{file=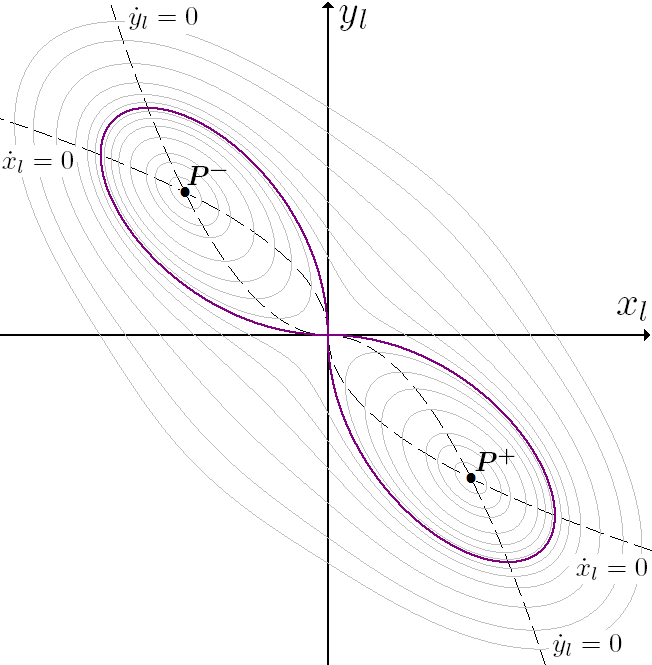, height = 6 cm}}
\caption{If system (${\rm S}_{p^*}$) is autonomous, then it is Hamiltonian. It presents periodic solutions and two homoclinic trajectories, contained in the curve $H_{p^*}=0$. Isoclines $\dot x_l=0$ and $\dot y_l=0$ are also drawn.}
\label{figHham}
\end{figure}

\medbreak 

Now, let us assume \GA with $l\neq p^*$.
By \eqref{Po-rel}, we obtain
\begin{multline}\label{Hl-est}
\frac{d}{dt} H_l(x_l(t), y_l(t)) = \frac{d}{dt} \left[ \eu^{(\alpha_l+\gamma_l)t} \, H_{p^*}(x_{p^*}(t),y_{p^*}(t))  \right]\\
= (\alpha_l+\gamma_l) H_l(x_l(t), y_l(t), t) + \eu^{(\alpha_l+\gamma_l)t} \frac{\partial}{\partial t} \, G_{p^*}(t,x_{p^*}(t)) \,.
\end{multline}
We stress that, in this case, (${\rm S}_{p^*}$) is non-autonomous.
We introduce the function $\mathcal G_l(x)= G_l(x)/x|x|^{p-1}$. 
Since $g_l(x)/|x|^{p-1}$ is increasing, then $\mathcal G_l$ is increasing too: indeed, one has $\mathcal G_l'(x)=\frac{1}{|x|^{p+1}}\left[ x g_l(x) - p\, G_l(x) \right]$ which is non-negative, since
$$
G_l(x) = 
\int_0^x \frac{g_l(s)}{|s|^{p-1}} |s|^{p-1} \,ds \leq \frac{g_l(x)}{|x|^{p-1}}  \int_0^x |s|^{p-1}\,ds = \frac{x \,g_l(x)}{p}\,.
$$
By \eqref{Gll}, we obtain
\begin{eqnarray*}
\frac{\partial}{\partial t} G_{p^*}(t,x)
&=& \frac{\partial}{\partial t} \left( G_l(x \eu^{-(\alpha_{p^*}-\alpha_l)t}) \eu^{p(\alpha_{p^*}-\alpha_l)t} \right) \\
&=& \frac{\partial}{\partial t} \left(\mathcal G_l(x \eu^{-(\alpha_{p^*}-\alpha_l)t}) x|x|^{p-1} \right) \,,
\end{eqnarray*}
which has the sign of $\alpha_l -\alpha_{p^*}$:
$$
\alpha_l -\alpha_{p^*} \ 
\begin{cases}
> 0 & \text{ if } p_*<l<p^*\,,
\\
< 0 & \text{ if } l>p^*\,.
\end{cases} 
$$
Hence, by \eqref{alfagamma} and \eqref{Hl-est},
$$
\frac{d}{dt} H_l(\xl(t)) \ 
\begin{cases}
> 0 & \text{ if } p_*<l<p^*\,
\\
< 0 & \text{ if } l>p^*
\end{cases} 
\qquad
\text{ when } H_l(\xl(t))\geq 0 \,.
$$

\medbreak

Let us consider the case $l>p^*$
 and fix $\Q\in\RR^2\setminus\{{\bs O}\}$. Suppose that $\lim_{t\to-\infty} \xl(t,0,\Q)=\O$ (in particular  $\lim_{t\to-\infty} H_{l}(\xl(t,0,\Q))=0$) then, by the previous computation, $H_{l}(\xl(t,0,\Q))<0$ for every $t\in\RR$ so that $\lim_{t\to+\infty} \xl(t,0,\Q)=\bs{P^\pm}$.
Conversely, if $\lim_{t\to+\infty} \xl(t,0,\Q)=\O$, since $\alpha_l+\gamma_l\neq 0$, thanks to the Poincar\'e-Bendixson criterion, $\xl$ is not a homoclinic and there are not heteroclinic cycles. Therefore, $\lim_{t\to-\infty} |\xl(t,0,\Q)|=+\infty$.

Arguing similarly, in the case $p_*<l<p^*$, if $\lim_{t\to+\infty} \xl(t,0,\Q)=\O$ then $\lim_{t\to-\infty} \xl(t,0,\Q)=\bs{P^\pm}$, while if we assume $\lim_{t\to-\infty} \xl(t,0,\Q)=\O$ then $\lim_{t\to+\infty} |\xl(t,0,\Q)|=+\infty$.

Finally, introducing polar coordinates, it is immediately verified that the angular velocity of the solutions is unbounded as $|\xl|\to+\infty$, so that, if $|\xl(t,0,\Q)| \to +\infty$ then the trajectory draws an infinite number of rotations around the origin having the shape of a spiral.

Hence we can draw the stable and unstable manifolds for the autonomous system \eqref{sist}, when \GA is satisfied with $l\neq p^*$, as in Figure \ref{fignoncrit}. In particular, when $l>p^*$, the stable manifold has the shape of an unbounded double spiral.

\begin{figure}[t]
\centerline{\epsfig{file=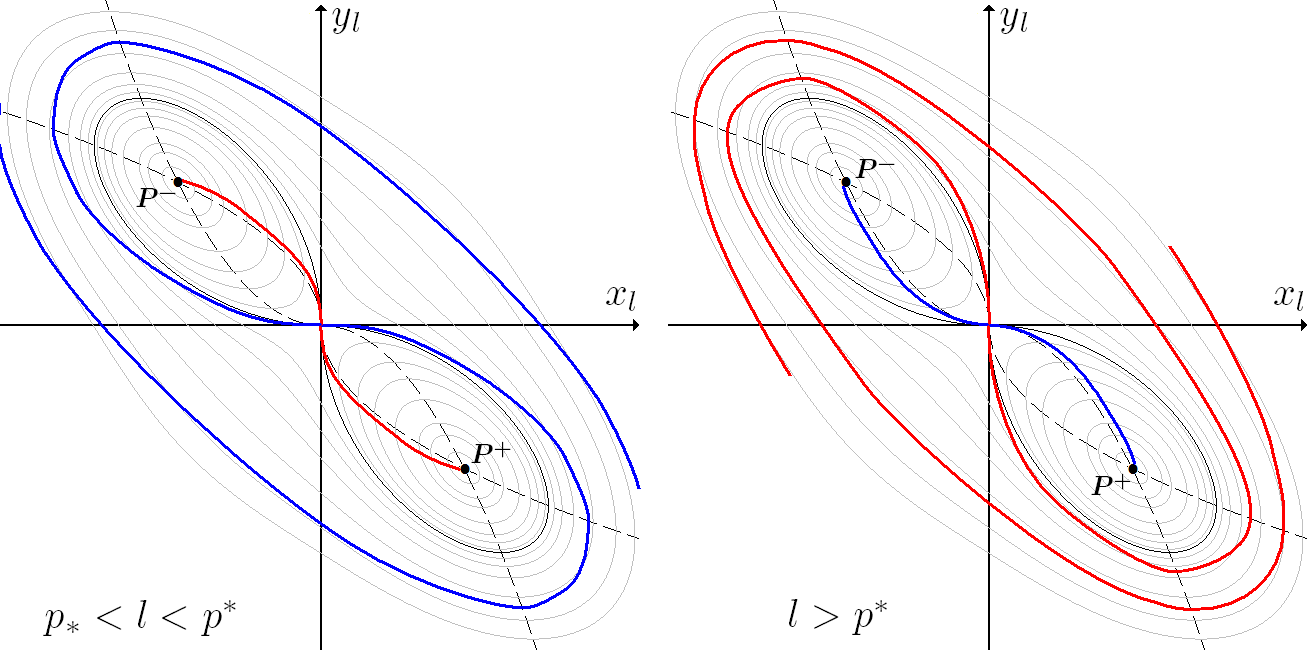, height = 5 cm}}
\caption{The unstable manifold and the stable manifold for autonomous system \eqref{sist} with $l\neq p^*$.}
\label{fignoncrit}
\end{figure}

The next remark underlines the correspondence between solutions of~\eqref{sist} converging to the origin and solutions of \eqref{plap.rad}.

 \begin{remark}\label{corrisp}
 Assume \GA. Consider the trajectory $\xl(t,\tau,\Q)$ of~\eqref{sist} and let $u(r)$ be the corresponding
solution of~\eqref{plap.rad}; then $u(r)$ is a regular solution if and only if $\Q \in M^u$, while it has fast decay if and only if
$\Q \in M^s$.
 \end{remark}
 
Such a result can be proved by standard arguments of invariant manifold theory, see e.g. \cite{DalFra,FraSfe1,FraSfe2}.
In fact $x_l(t) \propto \eu^{\alpha_l t}$, for $t\sim -\infty$, implies $u(r) \propto 1$ for $r\sim 0$ and
$y_l(t) \propto \eu^{\gamma_l t}$, for $t\sim +\infty$, implies $u'(r) \propto r^{-\frac{n-1}{p-1}}$ for $r\sim \infty$.
Moreover, the next remark gives the corresponding result for the non-trivial critical points $\bs{P^\pm}$.
 
\begin{remark}\label{corrP}
Assume \GA, with $l\neq p^*$. Consider the trajectory $\xl(t,\tau,\Q)$ of~\eqref{sist} and let $u(r)$ be the corresponding solution of~\eqref{plap.rad}.
Then, $u(r)$ is a slow decay solution if and only if $\lim_{t\to+\infty} \xl(t,\tau,\Q) = \bs{P^\pm}$.
Analogously, $u(r)$ is a singular solution if and only if
$\lim_{t\to-\infty} \xl(t,\tau,\Q) = \bs{P^\pm}$.
\end{remark}


\subsection{The non-autonomous case}\label{sub22}

In this section we provide the construction of invariant manifolds in the non-autonomous case.

The contents of this section collect only the hypotheses we need in the present paper, we refer to \cite{F2013DIE,FraSfe2} 
for an overview on this topic. 
In order to simplify the exposition we will treat the case $p\neq 2$ without further mentioning. The approach for the classical case $p=2$ gives a slightly different picture of the phase portrait since, in system \eqref{sist}, the term $y|y|^{\frac{2-p}{p-1}}=y$ is linear when $p=2$ (cf. \cite{DalFra,F2013DIE,FraSfe1}).

The following hypotheses on the function $g_l$ permit us to introduce stable manifolds for non-autonomous systems \eqref{sist}.

\begin{description}  
\item[\GS] Assume that there is $l>p_*$ such that
\begin{eqnarray*}
&& g_{l}(0,t)=\partial_x g_{l}(0,t) = 0\,, \qquad \text{ for any $t \in \RR$,}\\
&& \ds \lim_{t\to +\infty} g_{l}(x,t) = g^{+\infty}_{l}(x)
  \quad\text{and}\quad
  \lim_{t\to +\infty} \eu^{\varpi t} \, {\partial_t} \, g_{l}(x,t)=0
  \,,
\end{eqnarray*}
  uniformly on compact sets, where the function $g^{+\infty}_{l}$ is a non-trivial $C^1$ function satisfying \GA and $\varpi$ is a suitable positive constant.
\end{description}

Introducing a {\em time-type} variable $z= \eu^{-\varpi t}$, we obtain a $3$-dimensional autonomous system:
\begin{equation}\label{sist.na}
\left( \begin{array}{c}
\dot{x}_{l} \\
\dot{y}_{l} \\
\dot{z} \end{array}\right) = \left( \begin{array}{ccc} \alpha_{l} &
0 &0
\\ 0 & \gamma_{l} & 0 \\
0 & 0 & -\varpi
\end{array} \right)
\left( \begin{array}{c} x_{l} \\ y_{l} \\ z \end{array}\right) +\left(
\begin{array}{c} y_l|y_l|^{\frac{2-p}{p-1}} \\-
g_{l}(x_{l},-\ln(z)/\varpi)\\ 0\end{array}\right) .
\end{equation}
We have thus obtained an autonomous system in $\RR^3$ such that all its
 trajectories converge to the $z=0$ plane as $t \to +\infty$. Hence,~\eqref{sist.na} is useful to investigate
the asymptotic behavior of the solutions of \eqref{sist} in the future.
 Assume \GS. The origin admits a
$2$-dimensional stable manifold: we denote it by $\boldsymbol{W^s_l}$.
  From standard arguments of dynamical system
 theory, we see that  the set
 $W^s_{l}(\tau)=\boldsymbol{W^s_l} \cap \{ z= \eu^{-\varpi \tau} \}$ is a curve, for any $\tau \in \RR$,
 see e.g.~\cite{BaFloPino2000POIN,F2013DIE,FraSfe2,JPY}.
 
Let us denote by $W^s_l(+\infty)$ the stable manifold $M^s$ of the autonomous system \eqref{sist} where $g_l(x,t) \equiv g_l^{+\infty}(x)$. Then we have the following, cf. \cite{F2013DIE,JPY}.

\begin{remark}\label{allinfinitoWS}
  Assume \GS; then $W^s_{l}(\tau)$ approaches
  $W^s_{l}(+\infty)$ as $\tau \to +\infty$. 
  More precisely, if
  $W^s_{l}(\tau_0)$ intersects  transversally a certain line $L$ in a point $\Q(\tau_0)$ for $\tau_0 \in(-\infty,+\infty]$, then there is a neighborhood
  $I$  of $\tau_0$ such that $W^s_{l}(\tau)$ intersects $L$ in a point $\Q(\tau)$ for any $\tau \in I$, and $\Q(\tau)$ is continuous
  (in particular it is as smooth as $g_l$).
  \end{remark}
  
The proof is a consequence of standard facts in dynamical system theory (see e.g. \cite[\S 13]{CodLev} or \cite{JSell}). As a consequence we have the following characterization.
\begin{equation}\label{Ws}
    W^s_{l}(\tau):= \left\{ \Q \in \RR^2 \mid \lim_{t\to +\infty}  \boldsymbol{x_{l}}(t,\tau, \Q)=(0,0) \right\} \,.
\end{equation}

\medbreak

Arguing similarly it is possible to introduce unstable manifolds. For our purposes, we need to require a different behaviour for $g_l$ as $t\to-\infty$.

\begin{description}  
\item[\GU] Assume that there is $l>p_*$ such that
\begin{eqnarray*}
&& g_{l}(0,t)=\partial_x g_{l}(0,t) = 0\,, \qquad \text{ for any $t \in \RR$,}\\
&& \ds \lim_{t\to -\infty} g_{l}(x,t) = 0
  \quad\text{and}\quad
  \lim_{t\to-\infty} \eu^{-\varpi t} \, {\partial_t} \, g_{l}(x,t)=0
  \,,
\end{eqnarray*}
  uniformly on compact sets, where $\varpi$ is a suitable positive constant.
\end{description}

Arguing as above, denoting by $W^u_l(-\infty)$ the unstable manifold $M^u$ of the autonomous system \eqref{sist} where $g_l(x,t) \equiv 0$, we have the corresponding properties for the unstable manifold $\boldsymbol{W^u_l}$ and the curves $W^u_l(\tau) =\boldsymbol{W^u_l}\cap \{z=\eu^{\varpi\tau}\}$. Notice that, in this case, $M^u$ consists of the $x$ axis.

\begin{remark}\label{allinfinitoWU}
  Assume \GU; then $W^u_{l}(\tau)$ approaches
  $W^u_{l}(-\infty)$, i.e. the $x$ axis, as $\tau \to -\infty$. 
  More precisely, if
  $W^u_{l}(\tau_0)$ intersects  transversally a certain line $L$ in a point $\Q(\tau_0)$ for $\tau_0 \in[-\infty,+\infty)$, then there is a neighborhood
  $I$  of $\tau_0$ such that $W^u_{l}(\tau)$ intersects $L$ in a point $\Q(\tau)$ for any $\tau \in I$, and $\Q(\tau)$ is continuous
  (in particular it is as smooth as $g_l$).
  \end{remark}
  
As above, we can characterize the curves  $W^u_l(\tau)$ as follows:
\begin{equation}\label{Wu}
    W^u_{l}(\tau):= \left\{ \Q \in \RR^2 \mid \lim_{t\to-\infty} \boldsymbol{x_{l}}(t,\tau, \Q)=(0,0) \right\} \,.\\
\end{equation}

\medbreak

Let us list some properties of the manifolds we have introduced. See \cite{F2010JDE,F2013DIE,FraSfe2} 
for more details.

As in the autonomous case, we have the following correspondence between solutions of \eqref{plap.rad} and \eqref{sist}.

\begin{remark}\label{corr-non-aut} Assume $l>p_*$.
If $u(r,d)$ is a regular solution of \eqref{plap.rad}, then the corresponding trajectory $\xl(t)$ of \eqref{sist} satisfies $\xl(t)\in W^u_l(t)$ for every $t\in\RR$. Correspondingly, 
if $v(r,L)$ is a fast decay solution, then the corresponding trajectory satisfies $\xl(t)\in W^s_l(t)$ for every $t\in\RR$. Hence $\xl(t,\tau,\Q)$ is a regular fast decay solution of \eqref{plap.rad} if and only if $\Q\in W^u_l(\tau)\cap W^s_l(\tau)$.

Moreover, if a trajectory $\xl(t)$ of \eqref{sist} satisfies $\lim_{t\to+\infty} \xl(t) = \bs{P^\pm}$, then the corresponding solution $u(r)$ of \eqref{plap.rad} is a slow decay solution.
\end{remark}

The set $W^u_l(\tau)$ is tangent to the $x$-axis at the origin, while $W^s_l(\tau)$ is tangent to the $y$-axis at the origin (notice that, in the classical case $p=2$, the latter is tangent to the line $y+(n-2)x=0$).

The set $W^u_l(\tau)$ is split by the origin into two connected components, we will denote by $W^{u,+}_l(\tau)$ the one which
 leaves the origin and enters the $x>0$ semi-plane (corresponding to regular solutions which are positive for $r$ small), and by $W^{u,-}_l(\tau)$ the other which enters the $x<0$ semi-plane (corresponding to regular solutions
 which are negative for $r$ small). Similarly $W^{s}_{l}(\tau)$ is split by the origin into $W^{s,+}_{l}(\tau)$ and $W^{s,-}_{l}(\tau)$,
 which leave the origin and enter respectively in $x> 0$ and in $x< 0$ (corresponding to fast decay solutions 
 which are definitively positive and definitively negative respectively).

Following e.g. \cite[Lemma 4.2]{FraSfe2},
we introduce now some parametrizations of the manifolds. Fix $\tau\in\RR$ and consider the branch $W^{u,+}_l(\tau)$. For every $d>0$, there exists a point $\Q(d,\tau)\in W^{u,+}_l(\tau)$ corresponding to the regular solution $u(r,d)$ at $r=\eu^\tau$, i.e. $\Q(d,\tau)=(u(r,d)r^{\alpha_l},u'(r,d)|u'(r,d)|^{p-2} r^{\beta_l})$. Hence, we can find a parametrization $\Sigma_l^{u,+}(\cdot, \tau): (0,+\infty) \to \RR^2$ such that $\Sigma_l^{u,+}(d, \tau)=\Q(d,\tau)\in W^{u,+}_l(\tau)$. In particular, we have by construction that $\Sigma_l^{u,+}: (0,+\infty)\times \RR \to \RR^2$ is continuous. Similarly, we can introduce a continuous parametrization of $W^{s,+}_l(\tau)$, through the parameter $L$ associated to every fast decay solution $v(r,L)$, thus obtaining $\Sigma_l^{s,+}: (0,+\infty)\times \RR \to \RR^2$ such that $\Sigma_l^{s,+}(L,\tau)\in W^{s,+}_l(\tau)$ for every $L>0$ and $\tau\in\RR$. Again, $\Sigma_l^{u,-}$ and $\Sigma_l^{s,-}$ parametrize $W^{u,-}_l(\cdot)$ and $W^{s,-}_l(\cdot)$ respectively,  considering the regular solutions $u(r,-d)$ for every $d>0$ and the fast decay solutions $v(r,-L)$ for every $L>0$.

Moreover, we will consider the polar coordinates, associated to the sets $W^{u,\pm}_l(\tau)$ and $W^{s,\pm}_l(\tau)$ as follows:
 \begin{equation}\label{sigmap}
\begin{array}{lll}
\Sigma^{u,\pm}_{l}(d, \tau)& \!\!\!\!=  \rho^{u,\pm}_l(d, \tau) \big( \cos( \theta^{u,\pm}_l(d, \tau)), \sin( \theta^{u,\pm}_l(d, \tau))\big)\,, & d>0,\, \tau\in \RR,\\
\Sigma^{s,\pm}_{l}(L, \tau)& \!\!\!\!=  \rho^{s,\pm}_l(L, \tau) \big( \cos( \theta^{s,\pm}_l(L, \tau)), \sin( \theta^{s,\pm}_l(L, \tau))\big) \,, & L>0,\, \tau\in \RR.
\end{array}
\end{equation}
In particular we set for definiteness 
\begin{equation}\label{starting}
\begin{array}{ll}
\theta^{u,+}_l(0, \tau)=0 \,, & \theta^{u,-}_l(0, \tau)=-\pi \,, \\
\theta^{s,+}_l(0, \tau)=-\pi/2 \,, & \theta^{s,-}_l(0, \tau)=-3\pi/2 \,.
\end{array}
\end{equation}
Similarly we introduce polar coordinates associated to the solutions of \eqref{sist}:
\begin{equation}\label{xlpolar}
\xl(t,\tau,\Q) = \rho_l(t, \tau,\Q) \big( \cos( \phi_l(t, \tau,\Q)), \sin( \phi_l (t, \tau,\Q))\big)\,.
\end{equation}

\section{Proof of the main result}\label{sec3}

%
%
%

%

The proof is divided in five parts. At first we introduce a truncation of the nonlinearity $f$, and we prove the theorem for the truncated problem, introducing invariant manifolds and studying their shape. We show the existence of regular fast decay solutions looking for intersections between the unstable manifold and the stable one. The existence of regular slow decay solutions follows by topological arguments. Then, we discuss their nodal properties. Finally, we prove that all the solutions of the truncated equation are solutions of the original one, too.

\subsection{The truncated problem}\label{step1}
In order to ensure the continuability of the solutions of \eqref{plap.rad} we introduce the following truncation of the nonlinearity $f$:
\begin{equation}\label{trunc}
\bar f(u) =
\begin{cases}
f(u) & u\in (-d^-,d^+)\,, \\
0 & u\in (-\infty,-d^--1)\cup(d^++1,+\infty) \,,\\
 & \text{smooth and non-positive elsewhere}\,.
\end{cases}
\end{equation}
We will prove our main theorem for the truncated nonlinearity $\bar f$. Then, providing some a priori estimates, we will show that such solutions solve the original equation \eqref{plap}, too. Without loss of generality we assume
\begin{equation}\label{maxf}
\max_{\RR}|\bar f(u)| = f_\infty := \max_{u\in[-d^-,d^+]}|f(u)|\,,
\end{equation}


When we consider the truncation of $f$ introduced in \eqref{trunc}, applying the Fowler transformation \eqref{FowT} with $l=p\frac{q+\delta}{p+\delta}>p^*$, we obtain system \eqref{sist} with
\begin{equation}\label{truncg}
g_l(x,t)=
\begin{cases}
0 & \text{ if } x\eu^{-\alpha_l t}\leq -d^- \!-\!1 \,, \\
k(\eu^t) \, \bar f(x\eu^{-\alpha_l t}) \, \eu^{\alpha_l(l-1)t} & \text{ if } -d^- \!-\!1<x\eu^{-\alpha_l t}
<d^+ \!+\!1 \,,\\
0 & \text{ if } x\eu^{-\alpha_l t}\geq d^+ +1\,, 
\end{cases}
\end{equation}

\begin{lemma}
Assume the hypotheses of Theorem \ref{main}. The function $g_l$ in \eqref{truncg} satisfies both \GU and \GS with $l=p\frac{q+\delta}{p+\delta}$. 
\end{lemma}

\begin{proof}
Clearly, $g_l(0,t)=\partial_x g_l(0,t)=0$ for every $t\in\RR$.
Let us first prove that \GU holds.
Concerning the behavior of $g_l$ as $t\to -\infty$, we can find a constant $C>0$ such that
$$
g_l(x,t) \leq C \eu^{\alpha_l(q-1)t} \quad \text{and} \quad
|\partial_t g_l(x,t)| \leq C \eu^{\alpha_l(q-2)t} \,,
$$
for every $t\in\RR$ and $x\in\RR$. Indeed
$$
|g_l(x,t)| = h(\eu^t) \eu^{\delta t} \bar f(x \eu^{-\alpha_l t}) \eu^{\alpha_l(l-1)t} = h(\eu^t)  \bar f(x \eu^{-\alpha_l t}) \eu^{\alpha_l(q-1)t}
$$
and
\begin{multline*}
\partial_t g_l(x,t) = [h'(\eu^t)\eu^t +\alpha_l(q-1) h(\eu^t)] \bar f(x \eu^{-\alpha_l t}) \eu^{\alpha_l(q-1)t} \\ + h(\eu^t) x\bar f'(x \eu^{-\alpha_l t})\eu^{\alpha_l(q-2)t}
\end{multline*}
hold when $g_l(x,t) \neq 0$.
Hence, using \eqref{maxf}, we see that \GU holds and we get the existence of the unstable manifold $W^u_l(\tau)$ for every $\tau\in\RR$.

In order to prove the validity of \GS, let us now consider the limit as $t\to +\infty$.
We have
$$
g_l(x,t) = h(\eu^t)  x|x|^{q-2} \bar b(x \eu^{-\alpha_l t})\,, \qquad \text{when } g_l(x,t)\neq 0\,,
$$
where $\bar f(u) = u|u|^{q-2} \bar b(u)$. One has $\lim_{t\to+\infty} g_l(x,t) = h_\infty x|x|^{q-2} b(0)$, uniformly on compact set.
Moreover, 
\begin{multline*}
\eu^{\varpi t} \partial_t g_l(x,t) = h'(\eu^t)\eu^{(1+\varpi)t} x|x|^{q-2} b(x \eu^{-\alpha_l t}) \\ 
- \alpha_l h(\eu^t) x|x|^{q-2} b'(x \eu^{-\alpha_l t}) x \eu^{(\varpi-\alpha_l)t}\,, \qquad  \text{when } g_l(x,t)\neq 0
\end{multline*}
where $\varpi$ is given by \eqref{howisk}, suitably reduced in order to guarantee that  $\varpi<\alpha_l$.
Then \GS follows.
\end{proof}

\medbreak

Being $f(-d^-)=f(d^+)=0$, we have the existence of the constant solutions $u\equiv d^+$ and $u\equiv -d^-$ which correspond respectively to the trajectories
$\Eqp(t)=\xl(t,0,\bs{S_0^+})=(d^+ \, \eu^{\alpha_l t},0)$ and
$\Eqm(t)=\xl(t,0,\bs{S_0^-})=(-d^- \, \eu^{\alpha_l t},0)$.
In particular $\Eqp(t)\in W^{u,+}_l(t)$ and $\Eqm(t)\in W^{u,-}_l(t)$ for every $t\in \RR$.

By \eqref{howisf}, we have
\begin{equation}\label{unpogiu}
\begin{array}{l}
u'(r,d)<0 \text{ for $r$ small  if } d\in(0,d^+)\,,\\
u'(r,-d)>0 \text{ for $r$ small  if } d\in(0, d^-)\,.
\end{array}
\end{equation}
We consider the switched polar coordinates $\Omega^{u,\pm}_{l}(d, \tau)=( \theta^{u,\pm}_l(d, \tau) , \rho^{u,\pm}_l(d, \tau))$ of  $W^{u,+}_l(\tau)$, introduced as in \eqref{sigmap}, defined in the stripe $\stripe = \RR \times [0,+\infty)$. We introduce the sets
\begin{eqnarray*}
F^+(\tau)= \left\{ \Omega^{u,+}_{l}(d, \tau) \,:\, d\geq 0 \right\}
\,, \qquad
F^-(\tau)= \left\{ \Omega^{u,-}_{l}(d, \tau) \,:\, d\geq 0 \right\}\,,
\end{eqnarray*}
which are paths 
in the stripe $\stripe$.
Moreover
\begin{equation}\label{flow.x.axis}
x_l(t) \dot y_l(t) < 0 \text{ whenever } x_l(t)\in (-d^- \eu^{\alpha_l t}, d^+ \eu^{\alpha_l t})\setminus\{0\} \text{ and } y_l(t)= 0\,,
\end{equation} 
so that by \eqref{unpogiu}
we necessarily have
\begin{equation}\label{thu}
\theta^{u,+}_l(d, \tau) \leq 0 \text{ and } \theta^{u,-}_l(d, \tau) \leq -\pi \text{ for every } d>0 \text{ and } \tau\in\RR\,.
\end{equation}
Using \eqref{starting}, we can draw the sets corresponding to $W^u_l(\tau)$ for every $\tau\in\RR$ as follows.

\begin{proposition}\label{locateWu}
For every $\tau\in\RR$,
$F^+(\tau)$ connects the point $(0,0)$ to $\SEqp(\tau)=(0,d^+ \,\eu^{\alpha_l t})$, respectively
$F^-(\tau)$ connects  $(-\pi,0)$ to $\SEqm(\tau)=(-\pi,d^-\,\eu^{\alpha_l t})$ (see Figure \ref{figint}).
\end{proposition}

Arguing as above we can consider some curves associated to the sets $W^{s,+}_l(\tau)$ and $W^{s,-}_l(\tau)$. So, we denote by
\begin{eqnarray*}
\Omega^{s}_{l,2j}(L, \tau)=\big( \theta^{s,+}_l(L, \tau) - 2\pi j , \rho^{s,+}_l(L, \tau)\big)\,,\\
\Omega^{s}_{l,2j+1}(L, \tau)=\big( \theta^{s,-}_l(L, \tau) - 2\pi j , \rho^{s,-}_l(L, \tau)\big)\,,
\end{eqnarray*}
where $\Omega^{s}_{l,0}(\cdot, \tau)$ and $\Omega^{s}_{l,1}(\cdot, \tau)$ are respectively the natural representation of $W^{s,+}_l(\tau)$ and $W^{s,-}_l(\tau)$ on the stripe $\stripe$, by the choice \eqref{starting}. The others are simply their traslations of an angle $\Delta\theta=-2j\pi$.

By \GS with $l>p^*$, $W^s_l(+\infty)$ has the shape of a double spiral.
Therefore, for every $N\in\N$, there exists a compact set $\mathcal E \subset W^s_l(+\infty)$, containing the origin, such that both the branches $\mathcal E \cap W^{s,+}_l(+\infty)$ and $\mathcal E \cap W^{s,-}_l(+\infty)$ perform more than $N+1$ complete rotations in the plane. By Remark \ref{allinfinitoWS}, the unstable manifold $W^s_l(\tau)$ exists for any $\tau\in\RR$ and converges to $W^s_l(+\infty)$, as $\tau\to+\infty$. Therefore, $W^s_l(\tau)$ perform at least $N+1$ rotations for $\tau$ sufficiently large.
As a consequence we have the following remark.


\begin{remark}
For every integer $N$ we can find a time $\bar\tau_N$ with the following property:
\begin{equation}\label{Lpm}
\begin{array}{l}
\text{for every $\tau>\bar\tau_N$ there exist $L^+(N,\tau)>0$ and $L^-(N,\tau)>0$}\\
\qquad \text{such that $\theta^{s,\pm}_l(L^\pm(N,\tau), \tau)= 2\pi N + \pi/2$}
\end{array}
\end{equation}
(we assume that $L^\pm(N,\tau)$ is the minimum value with such a property). Moreover there exists $\chi_N>0$ such that
$$
\rho^{s,\pm}_l(L, \tau)< \chi_N\,, \quad \text{for every $L<L^\pm(N,\tau)$ and $\tau>\bar\tau_N$.}
$$
\end{remark}

Further, since
\begin{equation}\label{flow.y.axis}
\dot x_l(t) y_l(t) > 0 \text{ whenever } x_l(t)=0 \text{ and } y_l(t)\neq 0\,,
\end{equation}
we have
\begin{equation}\label{ths}
\theta^{s,+}_l(L, \tau)>-{\pi}/{2} \text{ and } \theta^{s,-}_l(L, \tau)>-3\pi/2 \text{ for every } L>0 \text{ and }\tau\in\RR \,.
\end{equation}

Let us choose $\tau_N\geq \bar\tau_N$ large enough to have
\begin{equation}\label{tauNlarge}
d^\pm \eu^{\alpha_l \tau_N} > \chi_N \,.
\end{equation}
We fix an integer $N$ and define 
$$
\begin{array}{rl}
E_{2j}(\tau) &= \left\{ \Omega^{s}_{l,2j}(L, \tau) \,:\, 0\leq L\leq L^+(N,\tau) \right\}\,,\\[2mm]
E_{2j+1}(\tau) &= \left\{ \Omega^{s}_{l,2j+1}(L, \tau) \,:\, 0\leq L\leq L^-(N,\tau) \right\}\,,
\end{array}
$$
for every $\tau>\tau_N$, corresponding to subsets of $W_l^s(\tau)$. The previous reasoning provide the following conclusion.

\begin{proposition}\label{locateWs}
For every integer $0\leq j \leq N$, the paths $E_{2j}(\tau)$ and $E_{2j+1}(\tau)$
intersect the line $\theta=\pi/2$ for every $\tau>\tau_N$.
Moreover, $E_k(\tau)\subset \RR \times [0,\chi_N]$ for every integer $k\in[0,2N+1]$ (see Figure \ref{figint}).
\end{proposition}

From \eqref{tauNlarge}, we have $\SEqpm(\tau)\notin \RR \times [0,\chi_N]$, for every $\tau>\tau_N$.
Hence, from Propositions \ref{locateWu} and \ref{locateWs}, we expect to find intersections between $F^\pm(\tau)$ and $E_k(\tau)$, cf. Figure \ref{figint}: the next section enters in such details.

\begin{figure}[t]
\centerline{\epsfig{file=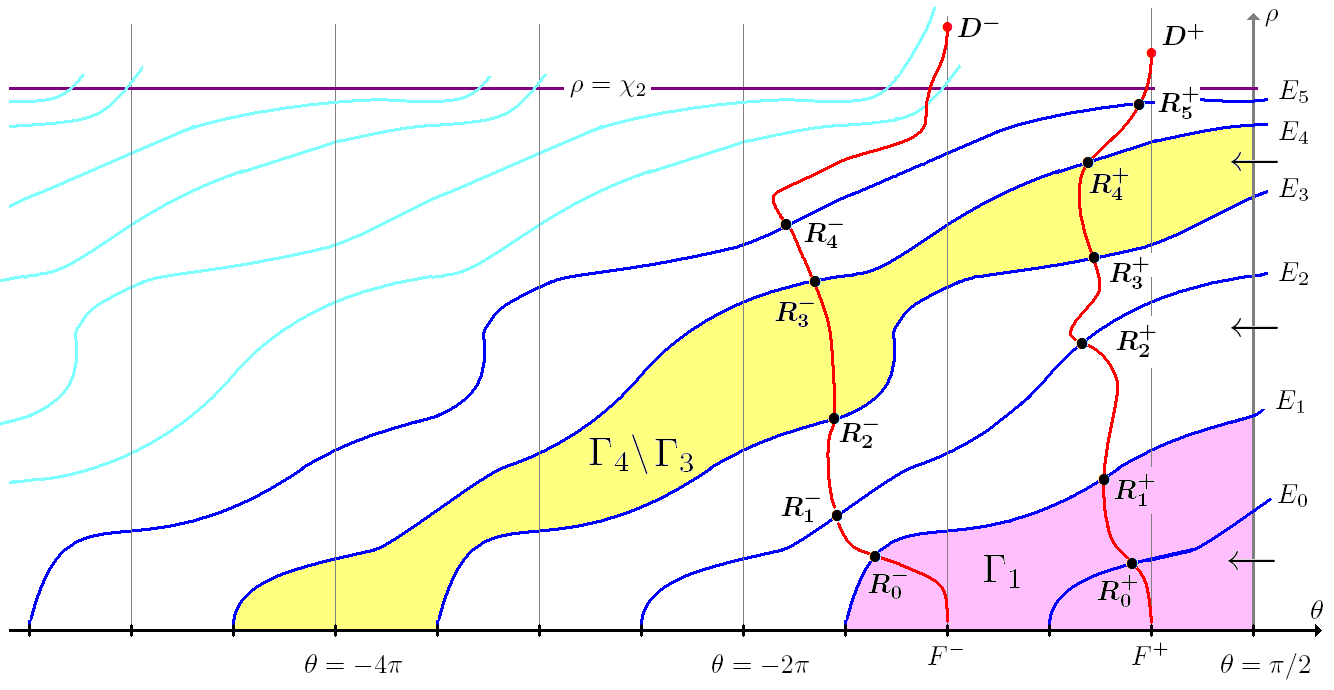, width = 12 cm}}
\caption{The figure shows, in the case $N=2$, the intersections $\bs{R^\pm_k}$ between $F^{\pm}(\tau)$ and $E_k(\tau)$, when $\tau>\tau_{N=2}$. The regions $\Gamma_1(\tau)$ and $\Gamma_4(\tau) \setminus \Gamma_3(\tau)$ have been colored.
 We have dropped the ``$(\tau)$'' for major clarity.}
\label{figint}
\end{figure}

\subsection{The existence of regular fast decay solutions}\label{step2}

Let us fix a positive integer $N$ and consider a time $\tau>\tau_N$, where $\tau_N$ is given by \eqref{tauNlarge}.
We are going to prove, the existence of intersections between $F^\pm(\tau)$ and $E_k(\tau)$ for $0\leq k \leq 2N+1$.

For every integer $k\in[0,2N+1]$, denote by $\Gamma_k(\tau)$ the region enclosed between $E_k(\tau)$, $\theta=\pi/2$ and $\rho=0$ (see Figure \ref{figint}).
Being $W^{s,+}_l(\tau) \cap W^{s,-}_l(\tau)=\{\O\}$, the paths $E_k(\tau)$ do not intersect each other. In particular, we have
\begin{equation}\label{Gamma-boxed}
\Gamma_0(\tau) \subset
\Gamma_1(\tau) \subset
\cdots \subset
\Gamma_{k-1}(\tau) \subset
\Gamma_{k}(\tau) \subset \cdots \,.
\end{equation}
Notice that the first part of the path $F^+(\tau)$ is contained in $\Gamma_k(\tau)$. Since $\SEqp(\tau)\notin \Gamma_k(\tau)$, $F^+(\tau)$ must leave $\Gamma_k(\tau)$ at a certain point. A similar reasoning can be done for $F^-(\tau)$: indeed, it starts inside $\Gamma_k(\tau)$ (except the case $F^-(\tau)\cap \Gamma_0(\tau)=\varnothing$) and $\SEqm(\tau)\in F^-(\tau)$ is such that $\SEqm(\tau)\notin \Gamma_k(\tau)$ (see Figure \ref{figint}). We denote by $\bs{R_k^\pm}(\tau)$ the first intersection (in the sense of the parameter~$d$) between the paths $F^\pm(\tau)$ and $E_k(\tau)$.
More precisely, we get the following lemma.

\begin{lemma}\label{lemint}
For every integer $N$, we can find constants $A_0,A_1,\ldots,A_{2N+1}$ and $B_0,B_1,\ldots,B_{2N}$ such that, for every $\tau> \tau_N$,
$$
\begin{array}{l}
\bs{R_{k}^+}(\tau) = \Omega^{u,+}_{l}(A_{k}, \tau) \in F^+(\tau) \cap E_k(\tau)\,, 
\\
\bs{R_{k}^-}(\tau) = \Omega^{u,-}_{l}(B_{k}, \tau) \in F^-(\tau) \cap E_{k+1}(\tau)\,, 
\end{array}
$$
(see Figure \ref{figint}).
We assume without loss of generality that they are the smallest positive constants with such a property (i.e. $F^\pm(\tau)$ exits from $\Gamma_k(\tau)$ for the first time at these points).
Correspondingly, $u(r,A_k)$ and $u(r,-B_k)$ are regular fast decay solutions.
\end{lemma}

The last assertion follows as an immediate consequence of Remark \ref{corr-non-aut}. Indeed, 
we can define also the corresponding points in the plane $(x_l,y_l)$:
$\bs{Q_{k}^+}(\tau) = \Sigma^{u,+}_{l}(A_{k}, \tau)$ and 
$\bs{Q_{k}^-}(\tau) = \Sigma^{u,-}_{l}(B_{k}, \tau)$. 
In particular
$\bs{Q_{k}^+}(\tau) \in W^{u,+}_l(\tau)\cap W^s_l(\tau)$ and
$\bs{Q_{k}^-}(\tau) \in W^{u,-}_l(\tau)\cap W^s_l(\tau)$ for every $\tau>\tau_N$,
so that they correspond to homoclinic orbit for system \eqref{sist}.

\medbreak

We have proved, for every positive integer $N$, the existence of intersections between $F^\pm(\tau)$ and $E_k(\tau)$, for $0\leq k \leq 2N+1$, when $\tau>\tau_N$. We stress that the integer $N$ can be chosen arbitrarily large (thus enlarging $\tau_N$ correspondingly), so that the previous constants $A_k$ and $B_k$ can be found for every choice of the integer $k$ as required by Theorem \ref{main}.
 The correct estimate on the number of nondegenerate zeroes will be provided in Section \ref{step4}.

\begin{figure}[t]
\centerline{\epsfig{file=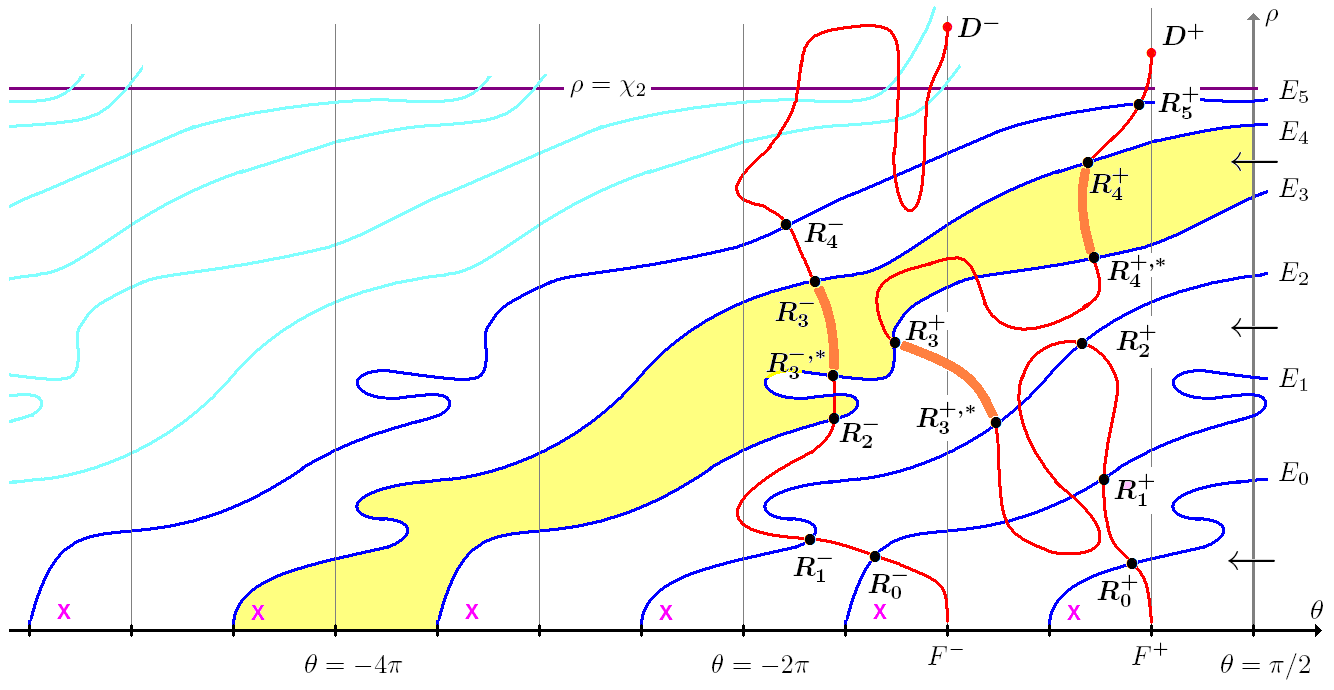, width = 12 cm}}
\caption{The figure shows, in the case $N=2$, the intersections $\bs{R^\pm_k}$ and $\bs{R^{\pm,*}_k}$ between $F^{\pm}(\tau)$ and $E_k(\tau)$, when $\tau>\tau_{N=2}$. Some of the $\bs{R^{\pm,*}_k}$ coincide with $\bs{R^\pm_{k-1}}$. When $\bs{R^{\pm,*}_k}\neq \bs{R^\pm_{k-1}}$ we have thickened the paths
$a_k^\pm(\tau)$ corresponding to slow decay solution.
Condition \eqref{flow.y.axis} provides a control on the corresponding trajectories $\Omega(t,\tau,\bar\Omega)$ which cannot pass the line $\theta=\pi/2$ from the left to the right. Moreover such solutions remain in $\Gamma_{k}(t)\setminus \Gamma_{k-1}(t)$ for every $t>\tau$, and converge to the switched polar coordinates of 
$\bs{P^\pm}$ denoted in the figure by {\sf x}.
We have dropped the ``$(\tau)$'' for major clarity.}
\label{figintdiff}
\end{figure}

\subsection{The existence of regular slow decay solutions}\label{step3}

We focus now our attention on the existence of slow decay solutions following the arguments presented in \cite{DalFra,FraSfe2}.

In the previous section, we have proved that $F^\pm(\tau)$ leaves $\Gamma_k(\tau)$ in some points $\bs{R_{k}^\pm}(\tau)$. In Figure \ref{figint}, a simple situation is pictured: the intersections $F^\pm(\tau) \cap E_k(\tau)$ consist of a unique point. Unfortunately, we cannot prove in general such a uniqueness property. Indeed, a more complex situation can arise: the intersections $F^\pm(\tau)\cap \Gamma_k(\tau)$ can consist of many disconnected paths, as in Figure \ref{figintdiff}. However, from \eqref{Gamma-boxed}, there exists necessarily at least one sub-path of $F^\pm(\tau)$ linking $E_{k-1}(\tau)$ to $E_k(\tau)$.

Hence, once fixed an integer $N$ and $\tau>\tau_N$ as above, for every integer $k\in(0,2N+1]$, we can find $A_{k}^*\in [A_{k-1},A_{k})$ such that
$$
A_k^* := \max \{d\in[0,A_k) \,:\, \text{exists } L >0 \text{ such that } \Omega^{u,+}_l(d,\tau)= \Omega^{s}_{l,k-1}(L,\tau) \} \,.
$$
Set $\bs{R^{+,*}_{k}} := \Omega^{u,+}_l(A_k^*,\tau)$ and
$a_k^+(\tau):=\{ \Omega^{u,+}_{l}(d, \tau) \,:\, d\in(A_{k}^*,A_{k}) \}$. Then $a_k^+(\tau) \subset\Gamma_{k}(\tau) \setminus \Gamma_{k-1}(\tau)$ 
(see Figure \ref{figintdiff}).
Similarly, for every integer $k\in(0,2N]$, we can find $B_{k}^*\in [B_{k-1},B_{k})$ such that
$$
B_k^* := \max \{d\in[0,B_k) \,:\, \text{exists } L >0 \text{ such that } \Omega^{u,-}_l(d,\tau)= \Omega^{s}_{l,k}(L,\tau) \} \,.
$$
Correspondingly, define $\bs{R^{-,*}_{k}} := \Omega^{u,-}_l(B_k^*,\tau)$ and
$a_k^-(\tau):=\{ \Omega^{u,-}_{l}(d, \tau) \,:\, d\in(B_{k}^*,B_{k}) \} \subset\Gamma_{k+1}(\tau) \setminus \Gamma_{k}(\tau)$.

Possibly we can have $A_k^*=A_{k-1}$ or $B_k^*=B_{k-1}$ for some $k$'s (in the simple situation presented in Figure \ref{figint}, they hold for every $k$).

\medbreak

\begin{lemma}\label{lemslow}
For every $d\in(A_k^*,A_k)$, $u(r,d)$ is a regular slow decay solution.
Similarly, for every $d\in(B_k^*,B_k)$, $u(r,-d)$ is a regular slow decay solution.
\end{lemma}

\begin{proof}
By Lemma \ref{corr-non-aut}, we need to prove that $\xl(t,\tau,\Q)\to \bs{P^+}$ or $\xl(t,\tau,\Q)\to \bs{P^-}$ as $t\to+\infty$,
where $\Q=\Sigma^{u,+}_l(d,\tau)$.


For every $\Q\in\RR^2$, let us denote by $\Omega(t,\tau,\Omega_\Q)$ the trajectory of $\xl(t,\tau,\Q)$ in the stripe $\stripe$, where $\Omega_\Q=(\phi_\Q,\rho_\Q)\in\stripe$ corresponds to the switched polar coordinates of $\Q$, i.e. $\Q= \rho_\Q (\cos \phi_\Q,\sin \phi_\Q)$.
Notice that
\begin{equation}\label{FEinvariant}
\begin{array}{l}
\Omega_\Q \in E_k(\tau) \Rightarrow \Omega(t,\tau,\Omega_\Q) \in E_k(t) \ \forall t\in\RR \,,\\
 \Omega_\Q \notin E_k(\tau) \Rightarrow \Omega(t,\tau,\Omega_\Q) \notin E_k(t) \ \forall t\in\RR \,, \\
\Omega_\Q \in F^\pm(\tau) \Rightarrow \Omega(t,\tau,\Omega_\Q) \in F^\pm(t) \ \forall t\in\RR \,,\\
 \Omega_\Q \notin F^\pm(\tau) \Rightarrow \Omega(t,\tau,\Omega_\Q) \notin F^\pm(t) \ \forall t\in\RR \,. \\
\end{array}
\end{equation}
Indeed, they correspond to the sets $W^{s,\pm}_l(\tau)$ and $W^{u,\pm}_l(\tau)$, which satisfy the same property, cf. \eqref{Ws} and \eqref{Wu}.

By \eqref{flow.y.axis}, the flow on $\{\theta=\pi/2\}\subset\stripe$ points towards $\{\theta<\pi/2\}$.
Hence, using also \eqref{FEinvariant}, we see that if $\Omega_\Q\in \Gamma_k(\tau) \setminus \Gamma_{k-1}(\tau)$ then
$\Omega(t,\tau,\Omega_\Q)\in \Gamma_k(t) \setminus \Gamma_{k-1}(t)$ for every $t>\tau$.
Moreover, since $W_l^{s}(t) \to W_l^s(+\infty)$, we have that $E_k(t) \to E_k(+\infty)$, where $E_k(+\infty)$ corresponds to a polar representation of a subset of $W^s(+\infty)$. In particular $\Gamma_k(t) \setminus \Gamma_{k-1}(t)$ remains bounded and  $\Gamma_k(t) \setminus \Gamma_{k-1}(t) \to \Gamma_k(+\infty) \setminus \Gamma_{k-1}(+\infty)$, where $\Gamma_j(+\infty)$ is the region enclosed by $E_j(+\infty)$, $\theta=\pi/2$ and $\rho=0$. 

We introduce the polar coordinates of the non-trivial critical points $\bs{P^\pm}= \rho_\pm (\cos\phi_\pm, \sin\phi_\pm)$
and the corresponding points $\Omega_{\bs{P},2j}=(\phi_+ -2\pi j,\rho_+)$ and $\Omega_{\bs{P},2j+1}=(\phi_- -2\pi j,\rho_-)$ on the stripe $\stripe$.
The unique attractor of $\Gamma_k(+\infty) \setminus \Gamma_{k-1}(+\infty)$ is $\Omega_{\bs{P},k}$, so we have $\Omega(t,\tau,\Omega_\Q)\to \Omega_{\bs{P},k}$ as $t\to+\infty$ for every $\Omega_\Q\in \Gamma_k(\tau) \setminus \Gamma_{k-1}(\tau)$.
The previous limits corresponds to $\xl(t,\tau,\Q) \to \bs{P^+}$ for $k$ even and 
to $\xl(t,\tau,\Q) \to \bs{P^-}$ for $k$ odd. Hence, the corresponding solution has a slow decay, cf. Remark \ref{corr-non-aut}.
Choosing $\Omega_\Q=\Omega^{u,\pm}_l(d,\tau) \in a_k^\pm(\tau)$ we find a regular slow decay solution $u(r,\pm d)$.
\end{proof}

As a consequence of the previous argument, we also have the following lemma.

\begin{lemma}\label{pri}
For every $d\in(A_k^*,A_k]$, $u(r,d)$  satisfies $-d^- < u(r,d) <d^+ $ for every $r > \eu^{\tau_N}$.
Similarly, for every $d\in(B_k^*,B_k]$, $u(r,-d)$ satisfies $-d^- < u(r,-d) <d^+ $ for every $r > \eu^{\tau_N}$.
\end{lemma}

\begin{proof}
For every $\tau>\tau_N$ we have $F^\pm(\tau) \subset \RR \times [0,\chi_N]$. So, by \eqref{tauNlarge}, we have $\rho_l^{u,\pm}(d,\tau)< d^\pm \, \eu^{\alpha_l \tau}$. Consequently, $-d^-<x_l(t,\tau,\Sigma_l^{u,\pm}(d,\tau)) \eu^{-\alpha_l t}<d^+$ for every $\tau>\tau_N$ and the assertion follows.
\end{proof}

%

\subsection{The number of non-degenerate zeros}\label{step4}
The correct estimates on the number of nondegenerate zeros of the fast decay solutions is given by the following lemma, see e.g. \cite[Lemma 3.3]{DalFra} (see also \cite{BaFloPino2000POIN,F2013DIE,FraSfe2} for a full fledged proof).

\begin{lemma}\label{anglecontrol}
Let us consider system~\eqref{sist} and assume \GU.
Consider the trajectory $\xl(\cdot ,\tau,\Q)$ with $\Q = \Sigma^{u,+}_{l}(d, \tau)$ and its polar coordinates \eqref{xlpolar}. Then,
the angle $\vartheta:=\theta^{u,+}_l(d,\tau)-\theta^{u,+}_l(0,\tau)=\theta^{u,+}_l(d,\tau)$ performed by
the unstable manifold $W_{l}^{u,+}(\tau)$
equals the angle $\phi:=\phi_l(t,\tau,\Q)-\phi_l(-\infty,\tau,\Q)$
performed by the trajectory $\xl(t,\tau,\Q)$ in the interval $(-\infty,\tau]$.

Similarly, assume \GS and
consider $\xl(\cdot ,\tau,\Q)$ with $\Q = \Sigma^{s,+}_{l}(L, \tau)$. Then,
the angle $\vartheta:=\theta^{s,+}_l(L,\tau)-\theta^{s,+}_l(0,\tau)=\theta^{s,+}_l(L,\tau)+\pi/2$ performed by
the stable manifold $W_{l}^{s,+}(\tau)$
equals, but with reversed sign, the angle $\varphi:=\phi_l(+\infty,\tau,\Q)-\phi_l(\tau,\tau,\Q)$
performed by the trajectory $\xl(t,\tau,\Q)$ in the interval $[\tau,+\infty)$.

A similar reasoning holds for $\Q = \Sigma^{u,-}_{l}(d, \tau)$ and $\Q = \Sigma^{s,-}_{l}(L, \tau)$.
\end{lemma}

The previous lemma is the key point in order to prove the nodal properties of the regular solutions we have found.

Let us start with regular fast decay ones. We consider $u(r,A_k)$ and the associated points $\bs{Q^+_k}=\Sigma^{u,+}_l(A_k,\tau)\in\RR^2$ and $\bs{R^+_k}=\Omega^{u,+}_l(A_k,\tau)\in\stripe$. By the previous lemma, in order to obtain the angle performed by $\xl(\cdot,\tau,\bs{Q^+_k})$ in the whole time interval $(-\infty,+\infty)$ we have to consider the angle variation along the path $F^+(\tau)$ between $(0,0)\in \stripe$ and $\bs{R^+_k}$ and then the one along $E_k(\tau)$ between $\bs{R^+_k}$ and $(-\pi(k+1/2),0)\in\stripe$. We easily obtain a tolal angle of $-\pi(k+1/2)$. So, by the flow condition \eqref{flow.y.axis}, the trajectory $\Omega(\cdot,\tau,\bs{R^+_k})$ intersects the vertical lines $\theta=-\pi/2-j\pi$ once for every integer $j\in\{1,\ldots,k\}$, thus finding exactly $k$ zeros of $\xl(\cdot,\tau,\bs{Q^+_k})$ which are non-degenerate.
A similar reasoning holds for $u(r,-B_k)$.

We turn now to consider $u(r,d)$ with $d\in(A^*_k,A_k)$ and, correspondingly, $\Q=\Sigma^{u,+}_l(d,\tau)$ and $\R=\Omega^{u,+}_l(d,\tau)$.
The total angle performed by the solution $\xl(\cdot,\tau,\Q)$ in the time interval $(-\infty,\tau)$ is given by the angle variation along the path $F^+(\tau)$. Then, for every $t>\tau$, $\Omega(t,\tau,\R)$ is forced to remain between the paths $E_{k-1}(t)$ and $E_k(t)$ and to converge to $\Omega_{\bs{P},k}$ as $t\to+\infty$: thus the total angle variation is 
$-2j\pi+\phi_+\in(-(2j+1)\pi,-2j\pi)$ if $k=2j$, resp $-2j\pi+\phi_-\in(-(2j+2)\pi,-(2j+1)\pi)$ if $k=2j+1$. By \eqref{flow.y.axis}, the trajectory $\Omega(\cdot,\tau,\bs{R})$ intersects all the lines $\theta= -\pi/2 - j\pi$ only once and correspondingly all the zeros are non-degenerate. 
A similar reasoning holds for $u(r,-d)$ with $d\in(B^*_k,B_k)$.

\subsection{Back to the orginal problem}\label{step5}

We have proved Theorem \ref{main} for the differential equation \eqref{plap} with $f$ replaced by its truncation $\bar f$ introduced in \eqref{trunc}.
We are going now to prove that such solutions solve also the original equation \eqref{plap}. 
To this aim we need Lemma \ref{apriori} below.

We underline that it is well-known in literature that, under hypotheses \eqref{howisf} and \eqref{howisk}, a positive solution of \eqref{plap.rad} with $u(0)<d^+$ necessarily satisfies $u(r)<d^+$ for every $r>0$ (and correspondingly negative ones with $u(0)>-d^-$ satisfy $u(r)>-d^-$). 
The situation is more complicated if we treat nodal solutions.
In the case of a decreasing weight $k$, we can easily provide the same {\em a priori} estimate by introducing an energy function. 

For a more general weight $k$, such a situation is not ensured and we argue as follows.


\begin{lemma}\label{apriori}
Consider $\Q=(Q_x,Q_y)\in\RR^2$ with $-d^- < Q_x \eu^{-\alpha_l \sigma} < d^+$ for a certain $\sigma$.
Suppose that there exists $\sigma'>\sigma$ such that $\xl(\sigma',\sigma,\Q)$ satisfies $-d^-<x_l(\sigma',\sigma,\Q) \eu^{-\alpha_l \sigma'} < d^+$.
Then, 
the solution $\xl(t,\sigma,\Q)$ of \eqref{sist}, with $g_l$ as in \eqref{truncg}, satisfies $-d^-<x_l(t,\sigma,\Q) \eu^{-\alpha_l t} < d^+$, for every $t\in(\sigma,\sigma')$.
\end{lemma}

\begin{proof}
Defining $\xi_l(t) = x_l(t) - d^+ \eu^{\alpha_l t}$ we obtain the system
\begin{equation}\label{sist.xi}
\begin{cases}
\dot \xi_l = \alpha_l \xi_l + y_l|y_l|^{\frac{2-p}{p-1}}\\
\dot y_l = \gamma_l y_l - h_l(t,\xi_l)
\end{cases}
\end{equation}
where $h_l(t,\xi_l)=g_l(t,\xi_l+d^+ \eu^{\alpha_l t})$. Notice that, by \eqref{truncg}, $h_l(t,\xi_l)\leq 0$ if $\xi_l>0$.
Moreover, as in \eqref{flow.y.axis},
\begin{equation}\label{flow.xi}
\dot \xi_l(t) y(t) > 0 \text{ whenever } \xi_l(t)=0 \text{ and } y_l(t)\neq 0\,,
\end{equation} 
In particular $\mathcal Q_1 = \{ (\xi,y) \,:\, \xi \geq 0\,, y\geq 0 \}$ and $\mathcal Q_2 = \{ (\xi,y) \,:\, \xi \geq 0\,, y\leq 0 \}$ are respectively positively and negatively invariant sets.

In this setting we have to prove that if a solution $(\xi_l(t),y_l(t))$ of \eqref{sist.xi} satisfies $\xi_l(\sigma)<0$ and $\xi_l(\sigma')<0$ then $\xi_l(t)<0$ for every $t\in(\sigma,\sigma')$. Arguing by contradiction, let $t_0\in(\sigma,\sigma')$ be such that $\xi_l(t_0)\geq 0$. If $y_l(t_0)\geq 0$ then $(\xi_l(t_0),y_l(t_0))\in \mathcal Q_1$ which is invariant in the future and we get a contradiction with $\xi_l(\sigma')<0$. Conversely, if $y_l(t_0)< 0$ then $(\xi_l(t_0),y_l(t_0))\in \mathcal Q_2$ which is invariant in the past and we get a contradiction with $\xi_l(\sigma)<0$.

The estimates with respect $d^-$ is analogous.
\end{proof}

We consider a regular solution $u(r,d)$ of \eqref{plap.rad} with $f=\bar f$ as in \eqref{trunc} and the corresponding solution $\xl$ of system \eqref{sist} (notice that $g_l$ is as in \eqref{truncg}).

If $d\in(-d^-,d^+)$, by \eqref{unpogiu}, we have $u(r,d)\in(-d^-,d^+)$ for every $r\in[0,\eu^{t_1}]$, with $\eu^{t_1}$ sufficiently small. Hence,
 we have $-d^- <x_l(t) \eu^{-\alpha_l t}<d^+ $ for every $t \leq t_1$.
%

Now, given a regular solution $u(r,d)$ with $d\in(A_k^*,A_k]$, setting $N>k/2$ and $\Q=\Sigma_l^{u,+}(d,\tau)$, by Lemma \ref{pri}, we have 
$-d^-<x_l(t,\tau,\Q) \eu^{-\alpha_l t}<d^+$ for every $t \geq \tau > \tau_N$.

So, we can apply Lemma \ref{apriori} with $\sigma=t_1$ and $\sigma'=\tau>\tau_N$, thus obtaining $-d^-<x_l(t,\tau,\Q) \eu^{-\alpha_l t}<d^+$ for every $t\in[t_1,\tau]$.
Summing up, the previous estimate holds for every $t\in\RR$.

\smallbreak

The same reasoning can be adapted to the case of a regular slow decay solution $u(r,-d)$ with $d\in(B_k^*,B_k]$. 

\smallbreak
 
Hence, the previously found solutions solve indeed the original equation and the proof of Theorem \ref{main} is thus completed.

 


\end{document}